\documentclass[oneside,english]{amsart}
\usepackage[T1]{fontenc}
\usepackage[latin9]{inputenc}
\usepackage{geometry}
\usepackage{babel}
\usepackage{amsbsy}
\usepackage{amstext}
\usepackage{amsthm}
\usepackage{amssymb}
\usepackage{graphicx}
\usepackage[all]{xy}
\usepackage[unicode=true,pdfusetitle,
 bookmarks=true,bookmarksnumbered=false,bookmarksopen=false,
 breaklinks=false,pdfborder={0 0 1},backref=false,colorlinks=false]
 {hyperref}
\usepackage{xcolor}
\usepackage{soul}
\makeatletter
\numberwithin{equation}{section}
\numberwithin{figure}{section}
\theoremstyle{plain}
\newtheorem{thm}{\protect\theoremname}
\theoremstyle{plain}
\newtheorem{cor}[thm]{\protect\corollaryname}
\theoremstyle{plain}
\newtheorem{lem}[thm]{Lemma}
\theoremstyle{definition}
\newtheorem{defn}[thm]{\protect\definitionname}
\newtheorem{rem}{\protect\remarkname}
\newcommand{\asplice}{\ensuremath{\makebox[0.3cm][c]{\raisebox{-0.3ex}{\rotatebox{90}{$\asymp$}}}}}

\author{Hongtaek Jung}
\address{Center for Geometry and Physics, Institute for Basic Science}
\email{htjung@ibs.re.kr}
\author{Sungkyung Kang}
\address{Center for Geometry and Physics, Institute for Basic Science}
\email{sungkyung38@icloud.com}
\author{Seungwon Kim}
\address{Center for Geometry and Physics, Institute for Basic Science}
\email{math751@ibs.re.kr}

\makeatother

\providecommand{\corollaryname}{Corollary}
\providecommand{\definitionname}{Definition}
\providecommand{\remarkname}{Remark}
\providecommand{\theoremname}{Theorem}

\begin{document}

\title{Concordance invariants and the Turaev genus}

\begin{abstract}
We show that the differences between various concordance invariants of knots, including Rasmussen's $s$-invariant and its generalizations $s_n$-invariants, give lower bounds to the Turaev genus of knots.  
Using the fact that our bounds are nontrivial for some quasi-alternating knots, we show the additivity of Turaev genus for a certain class of knots. This leads us to the first example of an infinite family of quasi-alternating knots with Turaev genus exactly $g$ for any fixed positive integer $g$, solving a question of Champanerkar-Kofman.
\end{abstract}
\maketitle

\section{Introduction}


A link diagram is alternating if its crossings alternate overpass and underpass when we travel along the diagram. An alternating link is a link that has an alternating link diagram. 

Alternating links have many interesting properties, with one of them being that their minimal crossing number is realized by their reduced alternating diagrams. This property was originally conjectured by Tait in the 1880's and proved by several mathematicians using the Jones polynomial \cite{kauffman1987state, murasugi1987jones, thistlethwaite1987spanning}. Turaev \cite{Tu1} gave a new proof of Tait's conjecture by introducing the \emph{Turaev surface}, which comes from the two extremal Kauffman states. 

It turns out that the Turaev surface has its own interesting properties. For instance, it is a Heegaard surface of $S^3$, and the link has an alternating projection which gives a disk decomposition of the Turaev surface. This allows us to consider every link as a generalized alternating link, i.e., an alternating link on some Heegaard surface. Moreover, the \emph{Turaev genus}, the minimal genus among all possible Turaev surfaces of a given knot, can be considered as a distance between a link and the set of alternating links. 
There are many previous works about computing Turaev genus \cite{champanerkar2007graphs, lowrance2008knot, abe2009turaev, dasbach2011turaev, DasbachLowrance}, but no general method to compute its exact value is known yet.

In this paper, we find an infinite family of new lower bounds of the Turaev genus. To state our result, we need to introduce a class $\mathcal{DL}$ of 
link concordance invariants. One can find its definition in Section \ref{mainsection}. We just remark here that this class $\mathcal{DL}$ contains the following well-known invariants.
\begin{enumerate}
    \item Rasmussen's $s$-invariant;
    \item $\frac{s_{n}}{1-n}$ for $n\ge 2$ where $s_n$ is the $\mathfrak{sl}_n$ link invariant with a suitable normalization;
    \item $-\sigma$, the negative of the link signature;
    \item $2\nu_0 -\ell +1$ where $\nu_0$ is a slice-torus link invariant in the sense of Cavallo-Collari and $\ell$ is the number of link components. In particular, $2\tau-\ell +1$, where $\tau$ is the Ozsv\'{a}th-Szab\'{o} $\tau$-invariant. 
\end{enumerate} 
Then our theorem on the lower bound to $g_T$ reads
\begin{thm}\label{theoremB}
Let $\mu,\nu$ be  in $\mathcal{DL}$.  Then for any knot $K$, we have 
\[
\frac{1}{2} | \mu(K) - \nu (K) | \le g_T(K).
\]
\end{thm}

Note that our lower bound covers the previous result of Dasbach-Lowrance \cite{dasbach2011turaev}.  

 One can consider the Turaev surface as a Jones polynomial theoretic, or more generally, Khovanov homology theoretic way to generalize alternating links. Then one can ask what is the Heegaard Floer homology theoretic generalization of alternating links.
 
 
 In Heegaard Floer homology, alternating links are very simple objects, since the branched double cover $\Sigma(K)$ of an alternating knot $K$ has the simplest possible $\widehat{HF}(\Sigma(K))$. In other words, $\Sigma(K)$ is an $L$-space. This result leads one to define a natural generalization of alternating links; we say that a link is quasi-alternating if it admits a diagram so that its two possible resolutions at a crossing are again quasi-alternating and the sum of the determinants of two resolutions is the same as the determinant of the given link. Note that for any quasi-alternating knot $K$, the branched double cover $\Sigma(K)$ is an $L$-space.
 
All alternating links are quasi-alternating. However, distinguishing quasi-alternating knots from alternating ones is a hard problem. One natural way to do that is to consider the Turaev genus of a quasi-alternating knot and see whether it is nontrivial. However, all known lower bounds to the Turaev genus vanish for quasi-alternating knots. Champanerkar and Kofman \cite{champanerkar2009twisting, champanerkar2014survey} asked whether the Turaev genus of quasi-alternating knots can take any non-negative integral value.



As a corollary of Theorem \ref{theoremB}, we answer the question of Champanerkar-Kofman in the affirmative way. More precisely, we will prove the following theorem.
\begin{thm}\label{theoremA}
For each integer $g>0$, there are infinitely many quasi-alternating knots whose Turaev genus equals $g$.
\end{thm}

\begin{rem}\label{turaevgenustwo}
Previously, the largest known Turaev genus of quasi-alternating knots was two. The Turaev genus two quasi-alternating knots can be identified by combining the work of Dasbach and Lowrance \cite{DasbachLowrance} and the work of Slavik Jablan \cite{jablan2014tables}. They are 11n95, 12n253, 12n254, 12n280, 12n323, 12n356, 12n375, 12n452, 12n706, 12n729, 12n811, 12n873.  
\end{rem}


We finish this section by giving a sketch of the proof of Theorem \ref{theoremA} assuming Theorem \ref{theoremB}. As a consequence of Theorem \ref{theoremB}, one can show the additivity of $g_T$ for a certain class of knots. 
\begin{cor}\label{theoremC}
 Let $K$ be a knot such that $s(K) + \limsup_{n\to \infty} \frac{s_n(K)}{n} \ge 2g_T(K)$. 
 Then 
 \[
 s(K) + \limsup_{n\to \infty} \frac{s_n(K)}{n} =2g_T(K).
 \]
 In particular, if $K$ and $L$ are knots with the prescribed property, then 
\[
g_T(L\sharp K) = g_T(L)+g_T(K).
\]
\end{cor}
We then argue that the 
pretzel knots $K_{p,q}=P(2p+1,-2q-1,2)$ with $p \ge q$ enjoy the assumption in Corollary \ref{theoremC}. Due to \cite[Theorem 3.2]{champanerkar2009twisting}, these knots $K_{p,q}$ are known to be quasi-alternating. Since being quasi-alternating is invariant under taking the connected sum, we may deduce Theorem \ref{theoremA} by setting $K=\sharp ^g K_{p,q}$.

\subsection*{Acknowledgements:} The authors would like to thank Lukas Lewark, Adam Lowrance, Peter Feller, and Ilya Kofman for helpful conversations. The authors owe Remark \ref{turaevgenustwo} to Adam Lowrance.

This work was supported by Institute for Basic Science (IBS-R003-D1).

\section{Preliminaries}

In this section, we review our main objects of interest, quasi-alternating knots and Turaev surfaces. 

\subsection{Quasi-alternating knots}
Quasi-alternating links were first introduced by Ozsv\'{a}th and Szab\'{o} \cite{ozsvath2005heegaard} to generalize alternating links in terms of Heegaard Floer theory.

\begin{defn}
The set of \emph{quasi-alternating links} $\mathcal{Q}$ is the smallest set of links such that:

\begin{enumerate}
    \item \{unknot\} $\subset \mathcal{Q}$;
    \item if a link $L$ contains a crossing $c$ so that the two smoothings $L_{0}$ and $L_{1}$ of $L$ at $c$ are in $\mathcal{Q}$ and $\det(L)=\det(L_{0})+\det(L_{1})$, then $L \in \mathcal{Q}$.
\end{enumerate}
\end{defn}

We remark that the connected sum of two quasi-alternating knots is again quasi-alternating. 

Ozsv\'{a}th and Szab\'{o} \cite{ozsvath2005heegaard} showed that the double branched cover of $S^3$ along a quasi alternating link is always an $L$-space, which is a rational homology $3$-sphere with the simplest Heegaard Floer homology. In fact, this $L$-space is interesting object since it is related to the taut foliation \cite{ozsvath2004holomorphic,bowden2016approximating,kazez2017c0} and left-orderability of fundamental group of $3$-manifolds \cite{boyer2013spaces}, which are fundamental objects in $3$-dimensional topology. 

\subsection{Turaev surfaces and Turaev genus}

Let $D\subset S^2$ be a link diagram of a link $L$. For each crossing \includegraphics[height=0.3cm, angle=90]{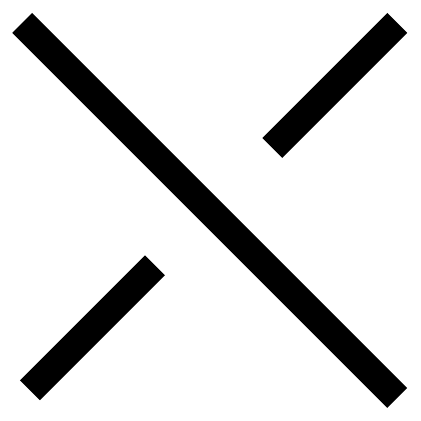} we can obtain the $A$-smoothing $\asplice$ or the $B$-smoothing $\asymp$ as in the Kauffman bracket.  If we smooth every crossing, then we get a set of simple loops on $S^2$ called a \emph{state}. 

There are two extreme cases of states. Suppose that we take the $A$-smoothing for each crossing. The resulting state is called the \emph{all-$A$ state} and is denoted as $s_A$. The other extreme is to take the $B$-smoothing for each crossing, which yields the \emph{all-$B$ state}, denoted as $s_B$.

The Turaev surface is a closed orientable surface obtained in the following way: first, we push $s_A$ to above $D$ and $s_B$ below $D$. Then we make a cobordism  from $s_A$ to $s_B$ as in Figure \ref{saddle2} by adding saddles for each crossing as in Figure \ref{saddle1}. The \emph{Turaev surface} $F(D)$ is then obtained by capping each boundary component of such cobordism off with a disk. The genus of $F(D)$ is called the \emph{Turaev genus of a diagram $D$}, denoted by $g_T(D)$. The minimal value of $g_T(D)$ over all possible diagrams $D$ of $L$ is called the \emph{Turaev genus} $g_T(L)$ of a link $L$.

\begin{figure}[!ht]
    \centering
    \includegraphics[width=0.5\textwidth]{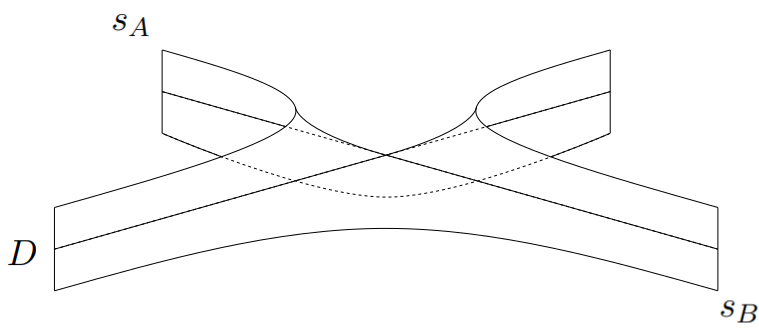}
    \caption{A local model of a Turaev surface near a saddle \cite{kim2019turaev}.}
    \label{saddle1}
\end{figure}
\begin{figure}[!ht]
    \centering
    \includegraphics[width=0.5\textwidth]{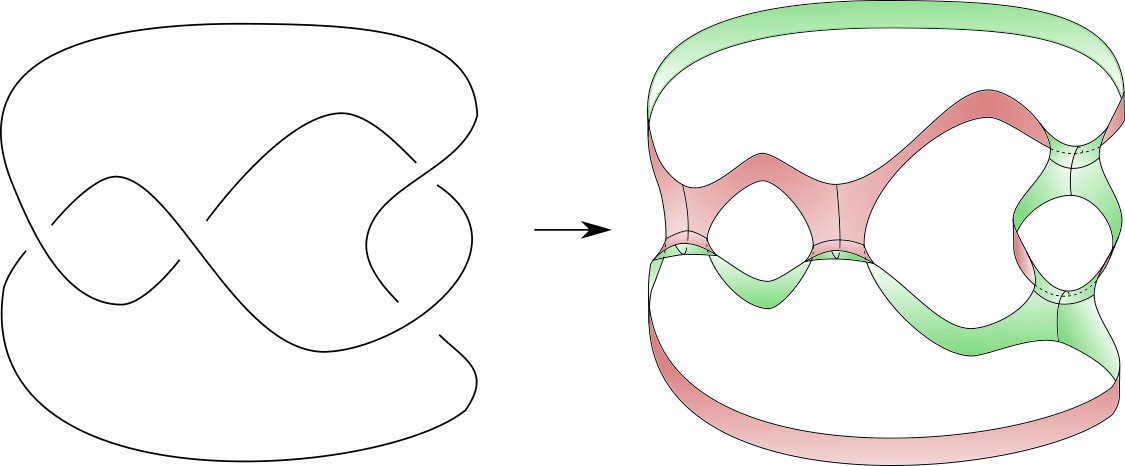}
    \caption{A Turaev surface, before capping-off \cite{kim2019turaev}.}
    \label{saddle2}
\end{figure}


We record some properties of the Turaev genus $g_T$ that are relevant to our discussion.
\begin{itemize}
    \item If $K$ is alternating, we have $g_T(K)=0$.
    \item For given two knots $K_1,K_2$, we have $g_T(K_1 \sharp K_2)\le g_T(K_1)+g_T(K_2)$ where $\sharp$ denotes the connected sum. 
\end{itemize}

Note that these properties make Turaev genus an alternating distance - a distance which measures how far a given link is from being alternating.

\section{Main results}\label{mainsection}

This section is devoted to proving our main results Theorem \ref{theoremB} and Theorem \ref{theoremA}. 
\begin{defn}
Let $L$ be an oriented link and $F:I\times I\rightarrow S^3$ be an embedding such that $L\cap F(I\times I)=F(I\times\partial I)$, where $I=[0,1]$ denotes the unit interval. Then one obtains a new link $L^{\prime}=(L\setminus F(I\times\partial I))\cup F((\partial I)\times I)$. If the orientation on $L\setminus F(I\times\partial I)$ induced by the given orientation of $L$ extends to $L^{\prime}$, then the extension is unique, so $L^{\prime}$ naturally becomes an oriented link, and we say that $L$ and $L^{\prime}$ are related by an \emph{oriented band surgery}.
\end{defn}
\begin{figure}[htb]
    \centering
    \includegraphics[width=0.5\textwidth]{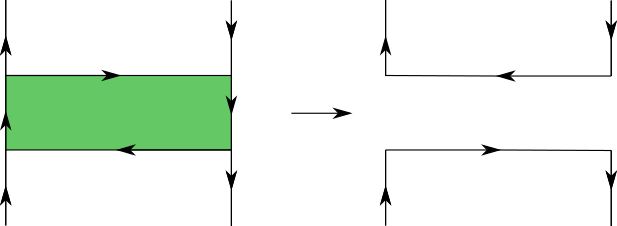}
    \caption{Oriented band surgery}
    \label{orientedbandsurgery}
\end{figure}

\begin{defn}
\label{DL}
An oriented link invariant $\nu$ which is invariant under overall orientation reversal is called a \emph{DL invariant} if it satisfies the following conditions.
\begin{enumerate}
    \item Suppose that an oriented link $L$ is obtained from a knot $K$ by performing $n$ oriented band surgeries. Then we have
    \[
    \vert \nu (K) - \nu (L) \vert \le n.
    \]
    \item For any positive (resp. negative) diagram $D$ of a non-split positive (resp. negative) link $L$,  
    \[
    s_B(D) - n_{-}(D) - 1  \leq \nu(L) \leq 1 + n_{+}(D) - s_A(D),
    \]
    where $s_A(D)$ and $s_B(D)$ denote the number of components in the all-A and all-B resolutions of $D$, respectively, and $n_{\pm}(D)$ denotes the number of positive (or negative) crossings in $D$.
\end{enumerate}
We denote by $\mathcal{DL}$ the set of DL invariants.
\end{defn}

Dasbach and Lowrance implicitly proved the following lemma in \cite{dasbach2011turaev} to get the lower bounds of the Turaev genus.

\begin{lem}[Dasbach-Lowrance]
\label{dl}
The concordance invariants $-\sigma$ and $s$ are DL invariants. Here, we are taking the extension of $s$ to links as defined by Beliakova and Wehrli in \cite{beliakova2008categorification}. Also, $2\tau$ satisfies the inequality in condition (2) in Definition \ref{DL} for knots.
\end{lem}

The fact that $-\sigma$ and $s$ are DL invariants can be generalized to the setting of slice-torus invariants. Cavallo and Collari defined the notion of slice-torus link invariants in \cite{cavallo2018slice}; we write their definition here for self-containedness.

\begin{defn}
A slice-torus link invariant is a real-valued link invariant $\nu$ which is invariant under strong link concordance, i.e. disjoint copies of smoothly and properly embedded cylinders in $S^3 \times I$, and satisfies the following conditions.
\begin{enumerate}
    \item If $L_1$ and $L_2$ are related by an oriented band surgery and $L_1$ has one component less than $L_2$, then 
    \[
    \nu(L_2)-1 \le \nu(L_1) \le \nu(L_2);
    \]
    \item $\nu$ is additive under split union;
    \item If $L$ is a $k$-component link, then 
    \[
    0 \le \nu(L) + \nu(-L^{\ast}) \le k,
    \]
    where $-L^{\ast}$ is the mirror image of $L$ with all orientations reversed;
    \item For any torus knot $K_{p,q}$, we have $\nu(K_{p,q}) = \frac{(p-1)(q-1)}{2}$.
\end{enumerate}
\end{defn}

Using the above definition, we can easily prove that any slice-torus link invariant naturally induces a DL invariant.

\begin{lem}
\label{dl-ineq}
Let $\nu_{0}$ be a slice-torus link invariant in the sense of Cavallo-Collari. Then the oriented link invariant $\nu = 2\nu_{0}- \ell +1$, where $\ell$ denotes the number of components, is a DL invariant.
\end{lem}

\begin{proof}
If $L_{1}$ and $L_{2}$ are related by an oriented band surgery and $\ell(L_{1}) = \ell(L_{2})-1$, then we have $\nu_{0} (L_{2})-1 \le \nu_{0}(L_{1}) \le \nu_{0}(L_{2})$ by the definition of slice-torus link invariants. Hence $\vert \nu(L_{1}) - \nu(L_{2}) \vert \le 1$, so we see that $\nu$ satisfies axiom (1) of DL invariants.
To prove axiom (2), let $L$ be a non-split positive (negative) link and $D$ be its positive (negative) diagram. Since the value of Cavallo-Collari slice-torus link invariants coincide for all positive links and negative links \cite[Theorem 1.2]{cavallo2018slice}, we only have to prove that the $s$-invariant satisfies the desired inequality, which was already given by the fact that $s$ is a DL invariant.
\end{proof}

\begin{cor}
The normalized $\mathfrak{sl}_n$ link invariants $\frac{s_{n}}{1-n}$ are DL invariants.
\end{cor}
\begin{proof}
$\frac{-s_{n}+(\ell -1)(n-1)}{2(n-1)}$ is a Cavallo-Collari slice-torus invariant for  each $n$, as mentioned in \cite[Example 2.4]{cavallo2018slice}.
\end{proof}



Now we are ready to prove that the Dasbach-Lowrance inequality also holds for all DL invariants.
\begin{lem}
\label{mainlem}
Let $K$ be a knot, $D$ be its diagram, and $\nu$ be a DL invariant. Then we have 
\[
{s}_{B}(D)-{n}_{-}(D)-1 \le \nu(K) \le 1+{n}_{+}(D)-{s}_{A}(D).
\]
\end{lem}
\begin{proof}
Note that $A$-smoothing is the oriented resolution for positive crossings. Consider the link diagram we get by $A$-smoothing all positive crossings of $D$ by $D_{-}$. Then we can form a graph $\Gamma$ whose vertices are connected components of $D_{-}$ and edges are positive crossings of $D$. Choose a spanning tree $T^{+}$ of $\Gamma$. Then we can $A$-smooth all positive crossings except those which form an edge of $T^{+}$ to get a new link diagram $D^{conn}_{-}$. This process is drawn for a case of knot $6_2$ in Figure \ref{Fig1}.

\begin{figure}[h]
\centering
\resizebox{.8\textwidth}{!}{\includegraphics{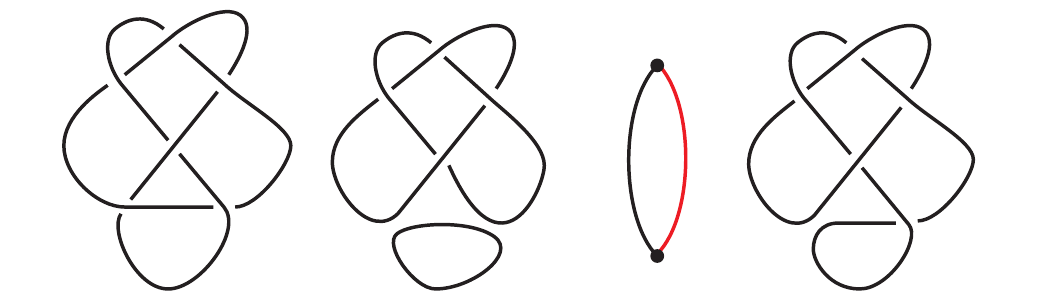}}
\caption{\label{Fig1} Left, a diagram of the knot $6_2$ with two positive crossings. Center-left, the negative diagram one gets by $A$-smoothing all positive crossings. Center-right, the graph $\Gamma$, where the red edge denotes the spanning tree $T^{+}$. Right, the diagram $D^{conn}_{-}$.}
\end{figure}

We claim that $D^{conn}_{-}$ is isotopic to a negative diagram. To prove this, we choose a component $D_{0}$ of $D_{-}$, which corresponds to a leaf of the tree $T^{+}$. Such a component must be \emph{innermost}, which means that the crossings of $D$ that correspond to edges of $T^{+}$ do not lie on any of the bounded components of $\mathbb{R}^{2} \setminus D_{0}$. The component $D_{0}$ is then connected to exactly one positive crossing in $D^{conn}_{-}$, so we can simply untwist to remove that crossing because $D_{0}$ is innermost. Repeating this process inductively gives us a negative diagram $D^{\prime}$ which is isotopic to $D^{conn}_{-}$.

Since $D^{\prime}$ is negative and connected, we have $n_{+}(D^{\prime})=0$, and the link $K^{-}$ represented by $D^{\prime}$ is non-split; see, for instance, \cite[Corollary 3.1]{cromwell1989homogeneous}. Also, by the construction of $D^{\prime}$, the all-$A$ smoothing of $D^{\prime}$ is isotopic to the smoothing of $D$ one gets by $A$-smoothing all crossings except those arising as edges of $T^{+}$ and then $B$-smoothing the rest. Hence we have $s_A(D^{\prime})=s_A(D)-\sharp E(T^{+})$, where $\sharp E(T^{+})$ denotes the number of edges in $T^{+}$.

Now we can prove the given inequality. Note that the link $K^{-}$ is formed by performing $n_{+}(D)-\sharp E(T)$ oriented band surgeries on $K$. Thus there exists an oriented cobordism of Euler characteristic $-n_{+}(D)+\sharp E(T^{+})$ between $K$ and $K^{-}$, so we have the inequality 
\[
\vert \nu(K)-\nu(K^{-}) \vert \le n_{+}(D)-\sharp E(T^{+}).
\]
Furthermore, since $K^{-}$ is a non-split negative link and $D'$ is a negative diagram representing it, we also have 
\[
\nu(K^{-}) \le 1+n_{+}(D')-s_{A}(D')=1-s_{A}(D)+\sharp E(T^{+}).
\]
Therefore we get 
\[
\nu(K) \le \nu(K^{-}) + n_{+}(D)-\sharp E(T^{+}) \le 1+n_{+}(D)-s_{A}(D),
\]
which proves the upper bound. 

 To prove the lower bound, we simply dualize our arguments. Instead of $A$-smoothing positive crossings of $D$, we can $B$-smooth its negative crossings to construct a positive link diagram $D_{+}$. Then we can apply the spanning tree argument to construct a tree $T^{-}$ and a connected positive diagram $D^{conn}_{+}$. Denote the non-split link represented by the diagram $D^{conn}_{+}$ as $L^{+}$. Then $L^{+}$ is obtained from $K$ by performing $n_{-}(D)-\sharp E(T^{-})$ oriented band surgeries. Therefore we get 
\[
\nu (K) \ge \nu (L^{+}) - n_{-}(D) + \sharp E(T^{-}) \ge 1-n_{-}(D)+s_{B}(D),
\]
completing the proof.
\end{proof}

\begin{proof}[Proof of Theorem \ref{theoremB}] 
This follows from Lemma \ref{mainlem} and the fact that 
$$1+{n}_{+}(D)-{s}_{A}(D) - ({s}_{B}(D)-{n}_{-}(D)-1) = g_T(D).$$
\end{proof}

\begin{proof}[Proof of Theorem  \ref{theoremA}]
For  any positive integers $p,q$ satisfying $p\ge q$, we consider the pretzel knot $K_{p,q}=P(2p+1,-2q-1,2)$. It is proven in Theorem 3.2 of \cite{champanerkar2009twisting} that such knots are always quasi-alternating, so its $s$-invariant is determined by its signature. More specifically, we have
\[
s(K_{p,q}) = -\sigma(K_{p,q}) = 2(p-q).
\]
On the other hand, Lewark \cite{lewark2014rasmussen} showed that the value $\frac{s_{n}(K_{p,q})}{1-n}$ can be either $2(p-q)-2$ or $2(p-q)-2+\frac{2}{n-1}$. Thus we have $1-\frac{1}{n-1} \le \frac{1}{2}\lvert s(K_{p,q}) - \frac{s_n (K_{p,q})}{1-n}\rvert$. So, by Theorem \ref{theoremB} and the additivity of both $s$ and $s_n$ under connected sums, we see that the following inequality holds for all integers $g\ge 1$ and $n\ge 2$, where $\sharp^{g}K_{p,q}$ denotes the connected sum of $g$ copies of $K_{p,q}$:
\[
g-\frac{g}{n-1} \le g_{T}(\sharp^{g}K_{p,q}).
\]
Taking the limit $n\rightarrow \infty$ on both sides gives $g\le g_{T}(\sharp^{g}K_{p,q})$.

From the direct application of the Turaev surface algorithm to the pretzel knot diagram of $K_{p,q}$, their Turaev genus are all 
one. By the subadditivity of Turaev genus, we get 
\[
g_{T}(\sharp^{g}K_{p,q}) \le g\cdot g_{T}(K_{p,q}) \le g.
\]
Therefore we get $g_{T}(\sharp^{g}K_{p,q})=g$. Since quasi-alternating knots are closed under connected sum, we obtained an infinite family of quasi-alternating knots with Turaev genus exactly $g$.
\end{proof}
\begin{rem}
One can show that all DL invariants take the same value for homogeneous knots, simply by mimicking the proof of \cite[Theorem 5]{lewark2014rasmussen}. Since alternating knots are homogeneous, our lower bounds for the Turaev genus vanish for alternating knots. In contrast, quasi-alternating knots are not always homogeneous, and for those knots, the $s_n$ invariants can behave differently from $s$, $\tau$, or $\sigma$. This is why our lower bounds can be nonzero for some quasi-alternating knots.
\end{rem}

\bibliographystyle{amsalpha}

\bibliography{ref2}

\providecommand{\bysame}{\leavevmode\hbox to3em{\hrulefill}\thinspace}
\providecommand{\MR}{\relax\ifhmode\unskip\space\fi MR }
\providecommand{\MRhref}[2]{%
  \href{http://www.ams.org/mathscinet-getitem?mr=#1}{#2}
}
\providecommand{\href}[2]{#2}
\begin{thebibliography}{BGW13}

\bibitem[Abe09]{abe2009turaev}
Tetsuya Abe, \emph{The {T}uraev genus of an adequate knot}, Topology Appl.
  \textbf{156} (2009), no.~17, 2704--2712.

\bibitem[BGW13]{boyer2013spaces}
Steven Boyer, Cameron~McA Gordon, and Liam Watson, \emph{On {L}-spaces and
  left-orderable fundamental groups}, Math. Ann. \textbf{356} (2013), no.~4,
  1213--1245.

\bibitem[Bow16]{bowden2016approximating}
Jonathan Bowden, \emph{Approximating {$C^0$}-foliations by contact structures},
  Geom. Funct. Anal. \textbf{26} (2016), no.~5, 1255--1296.

\bibitem[BW08]{beliakova2008categorification}
Anna Beliakova and Stephan Wehrli, \emph{Categorification of the colored
  {J}ones polynomial and {R}asmussen invariant of links}, Canad. J. Math.
  \textbf{60} (2008), no.~6, 1240--1266.

\bibitem[CC20]{cavallo2018slice}
Alberto Cavallo and Carlo Collari, \emph{Slice-torus concordance invariants and
  whitehead doubles of links}, Canad. J. Math. \textbf{72} (2020), no.~6,
  1423--1462. \MR{4176697}

\bibitem[CK09]{champanerkar2009twisting}
Abhijit Champanerkar and Ilya Kofman, \emph{Twisting quasi-alternating links},
  Proc. Amer. Math. Soc. \textbf{137} (2009), no.~7, 2451--2458.

\bibitem[CK14]{champanerkar2014survey}
\bysame, \emph{A survey on the {T}uraev genus of knots}, Acta Math. Vietnam.
  \textbf{39} (2014), no.~4, 497--514.

\bibitem[CKS07]{champanerkar2007graphs}
Abhijit Champanerkar, Ilya Kofman, and Neal Stoltzfus, \emph{Graphs on surfaces
  and {K}hovanov homology}, Algebr. Geom. Topol. \textbf{7} (2007), no.~3,
  1531--1540.

\bibitem[Cro89]{cromwell1989homogeneous}
P.~R. Cromwell, \emph{Homogeneous links}, J. London Math. Soc. (2) \textbf{39}
  (1989), no.~3, 535--552. \MR{1002465}

\bibitem[DL11]{dasbach2011turaev}
Oliver~T. Dasbach and Adam~M. Lowrance, \emph{Turaev genus, knot signature, and
  the knot homology concordance invariants}, Proc. Amer. Math. Soc.
  \textbf{139} (2011), no.~7, 2631--2645. \MR{2784832}

\bibitem[DL18]{DasbachLowrance}
\bysame, \emph{Invariants for {T}uraev genus one links}, Comm. Anal. Geom.
  \textbf{26} (2018), no.~5, 1103--1126. \MR{3900481}

\bibitem[Jab14]{jablan2014tables}
Slavik Jablan, \emph{Tables of quasi-alternating knots with at most 12
  crossings}, arXiv preprint arXiv:1404.4965 (2014).

\bibitem[Kau87]{kauffman1987state}
Louis~H Kauffman, \emph{State models and the {J}ones polynomial}, Topology
  \textbf{26} (1987), no.~3, 395--407.

\bibitem[KK19]{kim2019turaev}
Seungwon Kim and Ilya Kofman, \emph{{T}uraev surfaces}, arXiv preprint
  arXiv:1901.09995 (2019).

\bibitem[KR17]{kazez2017c0}
William Kazez and Rachel Roberts, \emph{{$C^0$} approximations of foliations},
  Geom. Topol. \textbf{21} (2017), no.~6, 3601--3657.

\bibitem[Lew14]{lewark2014rasmussen}
Lukas Lewark, \emph{{R}asmussen's spectral sequences and the
  $\mathfrak{sl}_{N}$-concordance invariants}, Adv. Math. \textbf{260} (2014),
  59--83.

\bibitem[Low08]{lowrance2008knot}
Adam Lowrance, \emph{On knot {F}loer width and {T}uraev genus}, Algebr. Geom.
  Topol. \textbf{8} (2008), no.~2, 1141--1162.

\bibitem[Mur87]{murasugi1987jones}
Kunio Murasugi, \emph{Jones polynomials and classical conjectures in knot
  theory}, Topology \textbf{26} (1987), no.~2, 187--194.

\bibitem[OS04]{ozsvath2004holomorphic}
Peter Ozsv{\'a}th and Zolt{\'a}n Szab{\'o}, \emph{Holomorphic disks and genus
  bounds}, Geom. Topol. \textbf{8} (2004), no.~1, 311--334.

\bibitem[OS05]{ozsvath2005heegaard}
\bysame, \emph{On the {H}eegaard {F}loer homology of branched double-covers},
  Adv. Math. \textbf{194} (2005), no.~1, 1--33.

\bibitem[Thi87]{thistlethwaite1987spanning}
Morwen~B Thistlethwaite, \emph{A spanning tree expansion of the {J}ones
  polynomial}, Topology \textbf{26} (1987), no.~3, 297--309.

\bibitem[Tur87]{Tu1}
V.~G. Turaev, \emph{A simple proof of the {M}urasugi and {K}auffman theorems on
  alternating links}, Enseign. Math. (2) \textbf{33} (1987), no.~3-4, 203--225.
  \MR{925987}

\end{thebibliography}
\end{document}